\documentclass[reqno]{amsart}

\usepackage{amsfonts,amsmath,amssymb,color}

 \newtheorem{theorem}{Theorem}
 \newtheorem{lemma}{Lemma}

\newcommand{\C}{\mathbb{C}}
\newcommand{\N}{\mathbb{N}}
\newcommand{\R}{\mathbb{R}}

\newcommand{\cB}{\mathcal{B}}
\newcommand{\cF}{\mathcal{F}}
\newcommand{\cM}{\mathcal{M}}
\newcommand{\cS}{\mathcal{S}}

\newcommand{\eps}{\varepsilon}
\newcommand{\fM}{\mathfrak{M}}
\begin{document}
\title[Fourier convolution operators]
{Fourier convolution operators\\ with symbols equivalent to zero
at infinity\\ on Banach function spaces}

\thanks{This work was partially supported by the Funda\c{c}\~ao para a
Ci\^encia e a Tecno\-lo\-gia (Portuguese Foundation for Science and Technology)
through the projects UID/ MAT/00297/2019 (Centro de Matem\'atica e
Aplica\c{c}\~oes).
The third author was also supported by the SEP-CONACYT Project A1-S-8793
(M\'exico).}

\author{Cl\'audio A. Fernandes, Alexei Yu. Karlovich, Yuri I. Karlovich}

\address{%
Cl\'audio A, Fernandes, Alexei Yu. Karlovich,
Centro de Matem\'atica e Aplica\c{c}\~oes,
Departamento de Matem\'atica,
Faculdade de Ci\^encias e Tecnologia,
Universidade Nova de Lisboa,
Quinta da Torre,
2829--516 Caparica,
Portugal}
\email{caf@fct.unl.pt}
\email{oyk@fct.unl.pt}

\address{Yuri I. Karlovich,
Centro de Investigaci\'on en Ciencias,
Instituto de Investigaci\'on en Ciencias B\'asicas y Aplicadas,
Universidad Aut\'onoma del Estado de Morelos,
Av. Universidad 1001, Col. Chamilpa,
C.P. 62209 Cuernavaca, Morelos, M\'exico}
\email{karlovich@uaem.mx}
\begin{abstract}
We study Fourier convolution operators $W^0(a)$ with symbols equivalent to
zero at infinity on a separable Banach function space $X(\R)$ such that the
Hardy-Littlewood maximal operator is bounded on $X(\R)$
and on its associate space $X'(\R)$. We show that the limit operators of
$W^0(a)$ are all equal to zero.
\end{abstract}
\keywords{Fourier convolution operator, Fourier multiplier, limit operator,
Banach function space, Hardy-Littlewood maximal operator,
equivalence at infinity.}
\maketitle
\section{Introduction}
The set of all Lebesgue measurable complex-valued functions on $\R$ is denoted
by $\fM(\R)$. Let $\fM^+(\R)$ be the subset of functions in $\fM(\R)$ whose
values lie  in $[0,\infty]$. The Lebesgue measure of a measurable set
$E\subset\R$ is denoted by $|E|$ and its characteristic function is denoted
by $\chi_E$. Following \cite[Chap.~1, Definition~1.1]{BS88}, a mapping
$\rho:\fM^+(\R)\to [0,\infty]$ is called a Banach function norm if,
for all functions $f,g, f_n \ (n\in\N)$ in $\fM^+(\R)$, for all
constants $a\ge 0$, and for all measurable subsets $E$ of $\R$, the following
properties hold:
\begin{eqnarray*}
{\rm (A1)} & & \rho(f)=0  \Leftrightarrow  f=0\ \mbox{a.e.}, \quad
\rho(af)=a\rho(f), \quad
\rho(f+g) \le \rho(f)+\rho(g),\\
{\rm (A2)} & &0\le g \le f \ \mbox{a.e.} \ \Rightarrow \ \rho(g)
\le \rho(f)
\quad\mbox{(the lattice property)},
\\
{\rm (A3)} & &0\le f_n \uparrow f \ \mbox{a.e.} \ \Rightarrow \
       \rho(f_n) \uparrow \rho(f)\quad\mbox{(the Fatou property)},\\
{\rm (A4)} & & |E|<\infty \Rightarrow \rho(\chi_E) <\infty,\\
{\rm (A5)} & & |E|<\infty \Rightarrow \int_E f(x)\,dx \le C_E\rho(f)
\end{eqnarray*}
with $C_E \in (0,\infty)$ which may depend on $E$ and $\rho$ but is
independent of $f$. When functions differing only on a set of measure zero
are identified, the set $X(\R)$ of all functions $f\in\fM(\R)$
for which $\rho(|f|)<\infty$ is called a Banach function space. For each
$f\in X(\R)$, the norm of $f$ is defined by
$\left\|f\right\|_{X(\R)} :=\rho(|f|)$.
Under the natural linear space operations and under this norm, the set
$X(\R)$ becomes a Banach space (see \cite[Chap.~1, Theorems~1.4 and~1.6]{BS88}).
If $\rho$ is a Banach function norm, its associate norm $\rho'$ is
defined on $\fM^+(\R)$ by
\[
\rho'(g):=\sup\left\{
\int_{\R} f(x)g(x)\,dx \ : \ f\in \fM^+(\R), \ \rho(f) \le 1
\right\}, \quad g\in \fM^+(\R).
\]
It is a Banach function norm itself \cite[Chap.~1, Theorem~2.2]{BS88}.
The Banach function space $X'(\R)$ determined by the Banach function norm
$\rho'$ is called the associate space (K\"othe dual) of $X(\R)$.
The associate space $X'(\R)$ is naturally identified with a subspace
of the (Banach) dual space $[X(\R)]^*$.

Let $\cF:L^2(\R)\to L^2(\R)$ denote the Fourier transform
\[
(\cF f)(x):=\widehat{f}(x):=\int_\R f(t)e^{itx}\,dt,
\quad
x\in\R,
\]
and let $\cF^{-1}:L^2(\R)\to L^2(\R)$ be the inverse of $\cF$.
It is well known that the Fourier convolution operator
\[
W^0(a):=\cF^{-1}a\cF
\]
is bounded on the space $L^2(\R)$ for every $a\in L^\infty(\R)$.
Let $X(\R)$ be a separable Banach function space. Then
by \cite[Lemma~2.12(a)]{KS14}, $L^2(\R)\cap X(\R)$ is dense in $X(\R)$.
A function $a\in L^\infty(\R)$ is called a Fourier multiplier on $X(\R)$ if the
convolution operator $W^0(a)$ maps $L^2(\R)\cap X(\R)$ into
$X(\R)$ and extends to a bounded linear operator on $X(\R)$. The function
$a$ is called the symbol of the Fourier convolution operator $W^0(a)$.
The set $\cM_{X(\R)}$ of all Fourier multipliers on  $X(\R)$ is a unital
normed algebra under pointwise operations and the norm
\[
\left\|a\right\|_{\cM_{X(\R)}}:=\left\|W^0(a)\right\|_{\cB(X(\R))},
\]
where $\cB(X(\R))$ denotes the Banach algebra of all bounded linear operators
on the space $X(\R)$.

Recall that the (non-centered) Hardy-Littlewood maximal operator $\cM$ of a
function $f\in L_{\rm loc}^1(\R)$ is defined by
\[
(\cM f)(x):=\sup_{J\ni x}\frac{1}{|J|}\int_J|f(y)|\,dy,
\]
where the supremum is taken over all finite intervals $J\subset\R$
containing $x$.

Let $V(\R)$ be the Banach algebra of all functions $a:\R\to\C$ with finite
total variation
\[
V(a):=\sup\sum_{i=1}^n|a(t_i)-a(t_{i-1})|,
\]
where the supremum is taken over all partitions
$-\infty<t_0<\dots<t_n<+\infty$
of the real line $\R$ and the norm in $V(\R)$ is given by
$\|a\|_{V}:=\|a\|_{L^\infty(\R)}+V(a)$.
\begin{theorem}\label{th:Stechkin}
Let $X(\R)$ be a separable Banach function space such that the
Hardy-Littlewood maximal operator $\cM$ is bounded on $X(\R)$ and on its
associate space $X'(\R)$. If $a\in V(\R)$, then the convolution operator
$W^0(a)$ is bounded on the space $X(\R)$ and
\begin{equation}\label{eq:Stechkin}
\|W^0(a)\|_{\cB(X(\R))}
\le
c_{X}\|a\|_V
\end{equation}
where $c_{X}$ is a positive constant depending only on $X(\R)$.
\end{theorem}

This result follows from \cite[Theorem~4.3]{K15a}. Inequality \eqref{eq:Stechkin}
is usually called the Stechkin type inequality
(see also \cite[inequality (2.4)]{KILH12}).

Following \cite[p.~140]{DKILH14}, two Fourier multipliers $c,d\in \cM_{X(\R)}$
are called equivalent at infinity if
\[
\lim_{N\to\infty}\left\|\chi_{\R\setminus[-N,N]}(c-d)\right\|_{\cM_{X(\R)}}=0.
\]
In the latter case we will write $c\stackrel{\cM_{X(\R)}}{\sim} d$.

The aim of this paper is to start the study of Fourier convolution operators
with symbols equivalent at infinity to well behaved symbols by the method of
limit operators in the context of Banach function spaces. We refer to
\cite{RRS04} for a general theory of limit operators and to
\cite{KILH12,KILH13a,KILH13b} for its applications to the study of Fourier
convolution operators with piecewise slowly oscillating symbols on Lebesgue
spaces with Muckenhoupt weights, constituting a remarkable example of Banach
function spaces.

For a sequence of operators $\{A_n\}_{n\in\N}\subset\cB(X(\R))$, let
$\operatornamewithlimits{s-\lim}\limits_{n\to\infty}A_n$
denote the strong limit of the sequence, if it exists.
For $\lambda,x\in\R$, consider the function
$e_\lambda(x):=e^{i\lambda x}$.
Let $T\in\cB(X(\R))$ and $h=\{h_n\}_{n\in\N}\subset(0,\infty)$ be a sequence
satisfying $h_n\to +\infty$ as $n\to\infty$. The strong limit
\[
T_{h}:=\operatornamewithlimits{s-\lim}_{n\to\infty}
e_{h_n}Te_{h_n}^{-1}I
\]
is called the limit operator of $T$ related to the sequence
$h=\{h_n\}_{n\in\N}$, if it exists.

\begin{theorem}[Main result]
\label{th:LO-convolution-equivalent-0}
Let $X(\R)$ be a separable Banach function space such that the Hardy-Littlewood
maximal operator $\cM$ is bounded on the space $X(\R)$ and on its associate
space $X'(\R)$. If $a\in\cM_{X(\R)}$ is such that
$a\stackrel{\cM_{X(\R)}}{\sim}0$, then for every sequence
$h=\{h_n\}_{n\in\N}\subset(0,\infty)$, satisfying $h_n\to+\infty$ as
$n\to\infty$, the limit operator of $W^0(a)$ related to the sequence
$h$ is the zero operator.
\end{theorem}

As usual, let $C_0^\infty(\R)$ denote the set of all infinitely differentiable
functions with compact support and let $\cS(\R)$ be the Schwartz space of
rapidly decreasing smooth functions. Finally, denote by $\cS_0(\R)$ the set of
all functions $f\in\cS(\R)$ such that their Fourier transforms $\cF f$ have
compact supports.

The paper is organized as follows. In Section~\ref{sec:mollification},
we discuss approximation by mollifiers in separable Banach function spaces
such that $\cM$ is bounded on $X(\R)$. In Section~\ref{sec:density},
we show that under the assumptions of the previous section,
the set $\cS_0(\R)$ is dense in the space $X(\R)$. Finally, in
Section~\ref{sec:proof}, we prove Theorem~\ref{th:LO-convolution-equivalent-0},
essentially using the density of $\cS_0(\R)$ in the space $X(\R)$.
\section{Mollification in separable Banach function spaces}
\label{sec:mollification}

The following auxiliary statement might be of independent interest.
\begin{theorem}\label{th:mollification}
Let $\varphi\in L^1(\R)$ satisfy $\int_\R\varphi(x)\,dx=1$
and
\begin{equation}\label{eq:mollification-1}
\varphi_\delta(x):=\delta^{-1}\varphi(x/\delta),\quad x\in\R,\quad\delta>0.
\end{equation}
Suppose that the radial majorant of $\varphi$ given by
$\Phi(x):=\sup\limits_{|y|\ge|x|}|\varphi(y)|$ belongs to $L^1(\R)$.
If $X(\R)$ is a Banach function space such that the Hardy-Littlewood
maximal operator $\cM$ is bounded on the space $X(\R)$, then
for all $f\in X(\R)$,
\begin{equation}\label{eq:mollification-2}
\sup_{\delta>0}\|f*\varphi_\delta\|_{X(\R)}
\le 
L\|f\|_{X(\R)},
\end{equation}
where $L:=\|\Phi\|_{L^1(\R)}\|\cM\|_{\cB(X(\R))}$
and $\|\cM\|_{\cB(X(\R))}$ denotes the norm of the sublinear operator
$\cM$ on the space $X(\R)$. If, in addition, the space $X(\R)$ is 
separable, then for all $f\in X(\R)$,
\begin{equation}\label{eq:mollification-2*}
\lim_{\delta\to 0^+}
\|f*\varphi_\delta-f\|_{X(\R)}=0.
\end{equation}
\end{theorem}
\begin{proof}
The idea of the proof is borrowed from \cite[Theorem~2.4]{RS08}.
By the proof of \cite[Lemma~5.7]{CF13}, for every
$f\in L_{\rm loc}^1(\R)$,
\begin{equation}\label{eq:mollification-3}
\sup_{\delta>0}|(f*\varphi_\delta)(x)|\le \|\Phi\|_{L^1(\R)}(\cM f)(x),
\quad
x\in\R.
\end{equation}
Inequality \eqref{eq:mollification-2} follows from inequality
\eqref{eq:mollification-3}, the boundedness of the Hardy-Littlewood 
maximal operator $\cM$ on the space $X(\R)$ and Axiom (A2).

Now assume that  the space $X(\R)$ is separable. Then
by \cite[Lemma~2.12(a)]{KS14},
the set $C_0^\infty(\R)$ is dense in the space $X(\R)$. Take $f\in X(\R)$ and
fix $\eps>0$. Then there exists $g\in C_0^\infty(\R)$ such that
\begin{equation}\label{eq:mollification-4}
\|f-g\|_{X(\R)}<\frac{\eps}{2(L+1)}.
\end{equation}
Hence for all $\delta>0$,
\begin{equation}\label{eq:mollification-5}
\|f*\varphi_\delta-f\|_{X(\R)}
\le
\|(f-g)*\varphi_\delta-(f-g)\|_{X(\R)}
+
\|g*\varphi_\delta-g\|_{X(\R)}.
\end{equation}
Taking into account 
inequalities \eqref{eq:mollification-2} and \eqref{eq:mollification-4}, 
we obtain for all $\delta>0$,
\begin{align}
\|(f-g)*\varphi_\delta-(f-g)\|_{X(\R)}
&\le
\|(f-g)*\varphi_\delta\|_{X(\R)}+\|f-g\|_{X(\R)}
\nonumber\\
&\le
(L+1)\|f-g\|_{X(\R)}<\eps/2.
\label{eq:mollification-6}
\end{align}
Let $\{\delta_n\}$ be an arbitrary sequence of positive numbers such that
$\delta_n\to 0$ as $n\to\infty$. Since $g\in C_0^\infty(\R)$, it follows
from \cite[Chap.~III, Theorem~2(b)]{S70} that
\begin{equation}\label{eq:mollification-7}
\lim_{n\to\infty}(g*\varphi_{\delta_n})(x)=g(x)
\quad\mbox{for a.e.}\quad x\in\R.
\end{equation}
In view of \eqref{eq:mollification-3}, we have for all $n\in\N$,
\begin{equation}\label{eq:mollification-8}
|(g*\varphi_{\delta_n})(x)|\le \|\Phi\|_{L^1(\R)}(\cM g)(x),\quad x\in\R.
\end{equation}
Since $g\in C_0^\infty(\R)\subset X(\R)$ and the Hardy-Littlewood maximal
operator $\cM$ is bounded on the space $X(\R)$, we see that $\cM g\in X(\R)$.
Then $\cM g$ has absolutely continuous norm because the Banach
function space $X(\R)$ is separable (see
\cite[Chap.~1, Definition~3.1 and Corollary~5.6]{BS88}).
It follows from \eqref{eq:mollification-7}--\eqref{eq:mollification-8}
and the dominated convergence theorem for Banach
function spaces (see \cite[Chap.~1, Proposition~3.6]{BS88}) that
\[
\lim_{n\to\infty}\|g*\varphi_{\delta_n}-g\|_{X(\R)}=0.
\]
Since the sequence $\{\delta_n\}$ is arbitrary, this means that one can find
$\delta_0>0$ such that for all $\delta\in(0,\delta_0)$,
\begin{equation}\label{eq:mollification-9}
\|g*\varphi_\delta-g\|_{X(\R)}<\eps/2.
\end{equation}
Combining \eqref{eq:mollification-5}, \eqref{eq:mollification-6}, and
\eqref{eq:mollification-9}, we see that for all $\delta\in(0,\delta_0)$
one has
\[
\|f*\varphi_\delta-f\|_{X(\R)}<\eps,
\]
which immediately implies \eqref{eq:mollification-2*}.
\qed
\end{proof}
\section{Density of the set $\cS_0(\R)$}
\label{sec:density}
\begin{lemma}\label{le:Schwartz-in-BFS}
Let $X(\R)$ be a Banach function space such that the Hardy-Little\--wood
maximal operator $\cM$ is bounded on $X(\R)$. Then $\cS(\R)\subset X(\R)$.
\end{lemma}

\begin{proof}
Suppose that $f\in\cS(\R)$. Then, in particular,
\[
\rho_0(f):=\sup_{x\in\R}|f(x)|<\infty,
\quad
\rho_1(f):=\sup_{x\in\R}|xf(x)|<\infty.
\]
By \cite[Example~2.1.4]{G14},
\begin{equation}\label{eq:Schwartz-in-BFS-1}
\frac{\chi_{\R\setminus[-1,1]}(x)}{|x|}
\le
\chi_{\R\setminus[-1,1]}(x)(\cM\chi_{[-1,1]})(x).
\end{equation}
Since the function $\chi_{[-1,1]}$ belongs to $X(\R)$
by Axiom (A4) and since the operator $\cM$
is bounded on the space $X(\R)$, we have $\cM\chi_{[-1,1]}\in X(\R)$. Let
$\psi(x)=|x|$. Then in view of \eqref{eq:Schwartz-in-BFS-1}
and Axiom (A2), we obtain
\begin{align*}
\|f\|_{X(\R)}
&\le
\left\|\chi_{[-1,1]}f\right\|_{X(\R)}
+
\left\|\chi_{\R\setminus[-1,1]}\psi f \cM\chi_{[-1,1]}\right\|_{X(\R)}
\\
&\le
\rho_0(f)\left\|\chi_{[-1,1]}\right\|_{X(\R)}
+
\rho_1(f)\left\|\cM\chi_{[-1,1]}\right\|_{X(\R)}.
\end{align*}
Thus, $f\in X(\R)$.
\qed
\end{proof}
\begin{theorem}\label{th:density-S0}
Let $X(\R)$ be a separable Banach function space such that the Hardy-Littlewood
maximal operator $\cM$ is bounded on $X(\R)$. Then the set $\cS_0(\R)$ is dense
in the space $X(\R)$.
\end{theorem}
\begin{proof}
Let $f\in X(\R)$. Fix $\eps>0$. By \cite[Lemma~2.12(a)]{KS14}, there
exists a function $g\in C_0^\infty(\R)$ such that
\begin{equation}\label{eq:density-S0-1}
\|f-g\|_{X(\R)}<\eps/2.
\end{equation}
Let
\[
{
\varrho(x):=\left\{\begin{array}{lll}
e^{1/(x^2-1)} &\mbox{if}& |x|<1,
\\
0 &\mbox{if}& |x|\ge 1,
\end{array}\right.
\quad
\varphi(x):=\frac{(\cF^{-1}\varrho)(x)}{\int_\R(\cF^{-1}\varrho)(y)\,dy},
\quad x\in\R.
}
\]
As $\varrho\in C_0^\infty(\R)\subset\cS(\R)$, it follows immediately from
\cite[Corollary~2.2.15]{G14} that $\varphi\in\cS_0(\R)$. For all $\delta>0$,
we define the family of functions $\varphi_\delta$ by
\eqref{eq:mollification-1}. Since $g\in C_0^\infty(\R)$ and
$\varphi_\delta\in\cS(\R)$, we infer from \cite[Proposition~2.2.11(12)]{G14}
that
\[
[\cF(g*\varphi_\delta)](x)
=
(\cF g)(x)(\cF\varphi_\delta)(x)
=
(\cF g)(x)(\cF\varphi)(\delta x),
\quad x\in\R.
\]
As $\cF\varphi$ has compact support, we conclude that
$\cF(g*\varphi_\delta)$ also has compact support. Thus
$g*\varphi_\delta\in\cS_0(\R)$ for every $\delta>0$. By
Lemma~\ref{le:Schwartz-in-BFS}, $g*\varphi_\delta\in X(\R)$.

By the definition of the Schwartz class
$\cS(\R)$, there are constants $C_n>0$ such that
\[
|\varphi(x)|\le C_n(1+|x|)^{-n},
\quad
x\in\R,
\quad
n\in\N\cup\{0\}.
\]
Then
\[
\Phi(x)=\sup_{|y|\ge|x|}|\varphi(y)|\le C_n\sup_{|y|\ge |x|}(1+|y|)^{-n}
=C_n(1+|x|)^{-n}
\]
for $x\in\R$ and $n\in\N\cup\{0\}$. This estimate implies that the radial
majorant $\Phi$ of the function $\varphi$ is integrable.

Since $\Phi\in L^1(\R)$, the space $X(\R)$ is separable, and the
Hardy-Littlewood maximal operator $\cM$ is bounded on $X(\R)$, it follows from
Theorem~\ref{th:mollification} that there is a $\delta>0$ such that
\begin{equation}\label{eq:density-S0-2}
\|g*\varphi_\delta-g\|_{X(\R)}<\eps/2.
\end{equation}
Combining \eqref{eq:density-S0-1} and \eqref{eq:density-S0-2}, we see that
for every $\eps>0$ there is a $\delta>0$ such that
$\|f-g*\varphi_\delta\|_{X(\R)}<\eps$. Since $g*\varphi_\delta\in\cS_0(\R)$,
the proof is completed.
\qed
\end{proof}
\section{Proof of Theorem~\ref{th:LO-convolution-equivalent-0}}
\label{sec:proof}
Fix a sequence $\{h_n\}_{n\in\N}\subset(0,\infty)$ such that $h_n\to +\infty$
as $n\to\infty$. For every function $f\in\cS_0(\R)$ there exists a segment
$K=[x_1,x_2]\subset\R$ such that $\operatorname{supp}\cF f\subset [x_1,x_2]$.
Therefore
\begin{align}
e_{h_n}W^0(a)e_{h_n}^{-1}f
&=
W^0[a(\cdot+h_n)]f=\cF^{-1}[a(\cdot+h_n)\chi_K]\cF f
\nonumber\\
&=
W^0(a\chi_{K+h_n})f,
\label{eq:LO-convolution-equivalent-0-2}
\end{align}
where $K+h_n=\{x+h_n:x\in K\}$.

Fix $\eps>0$. Without loss of generality we may assume that $f\ne 0$. As
$a\stackrel{\cM_{X(\R)}}{\sim}0$, there exists $N>0$ such that
\begin{equation}\label{eq:LO-convolution-equivalent-0-3}
\left\|\chi_{\R\setminus[-N,N]}a\right\|_{\cM_{X(\R)}}
<
\frac{\eps}{3c_X\|f\|_{X(\R)}},
\end{equation}
where $c_X>0$ is the constant from Stechkin's type inequality
\eqref{eq:Stechkin}. Since $h_n\to+\infty$
as $n\to\infty$, we conclude that there exists $n_0\in\N$ such that
for all $n>n_0$, one has
$K+h_n\subset(N,+\infty)\subset\R\setminus[-N,N]$.
Therefore, for $n>n_0$, we have
\begin{equation}\label{eq:LO-convolution-equivalent-0-4}
a\chi_{K+h_n}=\chi_{\R\setminus[-N,N]}a\chi_{K+h_n}.
\end{equation}
By Theorem~\ref{th:Stechkin}, for every $n>n_0$, we have
\begin{equation}\label{eq:LO-convolution-equivalent-0-5}
\left\|\chi_{K+h_n}\right\|_{\cM_{X(\R)}}
\le
c_X\left\|\chi_{K+h_n}\right\|_{V}= 3c_X.
\end{equation}
Combining \eqref{eq:LO-convolution-equivalent-0-2}--%
\eqref{eq:LO-convolution-equivalent-0-5}, we see that for $n>n_0$,
\begin{align*}
\left\|e_{h_n}W^0(a)e_{h_n}^{-1}f\right\|_{X(\R)}
&\le
\left\|\chi_{R\setminus[-N,N]}a\chi_{K+h_n}\right\|_{\cM_{X(\R)}}\|f\|_{X(\R)}
\\
&\le
\left\|\chi_{R\setminus[-N,N]}a\right\|_{\cM_{X(\R)}}
\left\|\chi_{K+h_n}\right\|_{\cM_{X(\R)}}\|f\|_{X(\R)}
<\eps.
\end{align*}
Hence, for every $f\in\cS_0(\R)$,
\[
\lim_{n\to\infty}\left\|e_{h_n}W^0(a)e_{h_n}^{-1}f\right\|_{X(\R)}=0.
\]
Since $\cS_0(\R)$ is dense in $X(\R)$
(see Theorem~\ref{th:density-S0}), the latter equality immediately
implies that
\[
\operatornamewithlimits{s-\lim}_{n\to\infty}e_{h_n}W^0(a)e_{h_n}^{-1}I=0
\]
on the space $X(\R)$
in view of
\cite[Lemma~1.4.1(ii)]{RSS11}.
\qed

\end{document}